\documentclass[a4paper,12pt]{amsart}

\usepackage{amssymb,amsbsy,amsmath,amsfonts,amssymb,amscd}
\usepackage{latexsym}
\usepackage{graphics}
\usepackage{color}
\input xy
\xyoption{all}

\theoremstyle{plain}
\newtheorem{thm}{Theorem}[section]
\newtheorem{theorem}[thm]{Theorem}

\newtheorem{lemma}[thm]{Lemma}

\newtheorem{proposition}[thm]{Proposition}
\theoremstyle{definition}

\newtheorem{defin}[thm]{Definition}

\newtheorem{example}[thm]{Example}

\newtheorem{question}[thm]{Question}

\numberwithin{equation}{section}

\newcommand{\sB}{{\mathcal B}}

\newcommand{\sR}{{\mathcal R}}

\newcommand{\PP}{\ensuremath{\mathbb{P}}}
\newcommand\LL{{\mathbb L}}

\newcommand{\AAA}{\ensuremath{\mathbb{A}}}
\newcommand{\CC}{\ensuremath{\mathbb{C}}}

\newcommand{\ZZ}{\ensuremath{\mathbb{Z}}}

\newcommand{\hol}{\ensuremath{\mathcal{O}}}

\newcommand\la{\lambda}

\newcommand\be{\beta}
\newcommand\Ga{\Gamma}

\newcommand{\ra}{\ensuremath{\rightarrow}}

\def\eea{\end{eqnarray*}}
\def\bea{\begin{eqnarray*}}

\newcommand\dual{\mathrel{\raise3pt\hbox{$\underline{\mathrm{\thinspace d
\thinspace}}$}}}
\newcommand\qe{\ifhmode\unskip\nobreak\fi\quad $\Box$}       

\def\BOX{\hfill\lower.5\baselineskip\hbox{$\Box$}}

\newtheorem{theo}{Theorem}[section]

\newtheorem{remarkk}[theo]{Remark}
\newenvironment{rem}{\begin{remarkk}\rm}{\end{remarkk}}

\newcommand{\Proof}{{\it Proof. }}

\title [Ueno-type is unirational]{On the unirationality of higher dimensional Ueno-type manifolds}

\author{Fabrizio Catanese, Keiji Oguiso, Alessandro Verra}

\address{}
\email{}

\address{} \email{}


\begin{document}

\thanks{The first author acknowledges support of the ERC 2013 Advanced Research Grant - 340258 - TADMICAMT. The first and second author
acknowledge KIAS for support during their visits in the quality of KIAS scholars. }

\maketitle

{\em Dedicated to Lucian Badescu with friendship and admiration on
the occasion of his 70-th birthday.}

\begin{abstract}
We prove the unirationality of the Ueno-type manifold $X_{4,6}$. $X_{4,6}$ is the minimal resolution of the quotient
of the Cartesian product $E(6)^4$, where $E(6)$ is the equianharmonic elliptic curve, by the diagonal action  of  a cyclic 
group of order 6 (having a fixed point on each copy of $E(6)$). We collect also other results, and discuss
several related open questions.
\end{abstract}

\section{Introduction}

Let $k$ be any field of characteristic $\not= 3$ containing a primitive third root of unity $\omega$, respectively a field of characteristic $\not= 2$ 
containing a primitive fourth root of unity $i$.

We shall work over $k$ unless otherwise stated. 

One of the standard normal forms for the function fields of elliptic curves is the following  normal form, defining an affine plane curve with equation
$$E_{\la} = \{ (x,y) | y^2 = x (x-1) (x-\la)\}  .$$

The curve $E_{\la} $ is said to be {\bf harmonic}  if the cross ratio $\la \in \{ -1,2,1/2\}$, and {\bf equianharmonic} if $\la $ 
is a primitive 6-th root of $1$, $\la \in \{ \eta, \eta^{-1}\}$, $\eta : = - \omega$. 

These curves admit automorphisms of respective orders $4,6$, having a fixed point.

In the harmonic case if we take   the above normal form with $\la = -1$, the automorphism of order $4$ is given by 

$$g_4:  (x , y)  \mapsto (-x,  i y).   $$ 

In the equianharmonic case it is easier to see the automorphism $g_6$ of order $6$ by changing the normal form to
$$E_{\eta} \cong  \{ (x,y) | y^2 = x^3 - 1\},  $$
so that
$$g_6:  (x , y)  \mapsto (\omega x,  - y),   $$
while using the Fermat normal form things are more complicated, 
$$E_{\eta} \cong  \{ (x,y) | y^3 = x^3 - 1\},  $$
and
$$g_6:  (x , y)  \mapsto ( 1/x, \eta y/x).   $$

Using instead the normal form where $E_{\eta}$ is birational to the singular plane curve
$$E'_{\eta} =  \{ (x,y) | y^6 = x^2 (x - 1)\},  $$
the automorphism takes the easier form
$$g_6:  (x , y)  \mapsto ( x,  \eta  y),   $$
and one sees immediately that the field of $g_6$-invariant rational functions is
the field $k(x)$.

Let  now 
$(x : y : z)$ be  homogeneous coordinates on ${\PP}^2$,  let
$$E(6) := \{ (x : y : z) | y^2z = x^3 -z^3 \} \subseteq {\PP}^2,$$
 be the projective model of the elliptic curve $E_{\eta}$, on which the automorphism $g_6$ acts by
$$g_6 (x : y : z) = (\omega x : -y :z)\\ $$
and similarly let 
$$E(4) :=  \{ (x : y : z) | y^2z = x (x^2 -z^2)  \} \subseteq {\PP}^2,$$
 be the projective model of the elliptic curve $E_{-1}$, on which the automorphism $g_4$ acts by
$$g_4 (x : y : z) = ( - x : i y :z)  .$$

When $k$ is the complex number field ${\CC}$, we have 
$$(E(6), g_6) \simeq 
(T_{\omega}, -\omega)\,\, ,$$ 
where $T_{\omega} = {\CC}/ ({\ZZ} + \omega{\ZZ})$ is the elliptic curve with period $\omega$, and $-\omega$ is the automorphism induced by  multiplication by $-\omega$ on ${\CC}$. 
We have such an isomorphism  because $g_6$ acts on the  holomorphic $1$-form $ dx /y$ via multiplication by   $-\omega$. 

Similarly $$(E(4), g_4) \simeq 
(T_{i},  i)\,\, .$$ 

Let now $(C, g)$ be either $(E(6), g_6)$ or $(E(4), g_4)$, and let $g$ act diagonally on the Cartesian product  $C^n.$ We set 
$$Z_n (4) :  = E(4)^n, \  Z_n (6) :  = E(6)^n.$$ 

\begin{defin}
We define the {\bf Ueno-type  manifold} $X_{n,6}$ of dimension $n$ to be the minimal resolution of singularities of the normal variety  $Y_{n,6} : = Z_n (6) / g_6$,
while we reserve the name of  {\bf Ueno-Campana manifold} for the Ueno-type manifold $X_{n,4}$ of dimension $n$, which is
 the minimal resolution of singularities of the normal variety  $Y_{n,4} : = Z_n (4) / g_4$.
\end{defin}

Observe that the quotient $n$-fold 
$$Y_{n,6} :=  Z_n (6) / g_6, \ n \geq 2, $$
has finitely many singular points of type 
$(1,1,1, \cdots, 1)/6$, of type $(1,1,1, \cdots, 1)/3$ and of type $(1,1,1, \cdots, 1)/2$, which are all  $k$-rational. $X_{n,6}$ is the blow up of $Y_{n,6}$
 at the maximal ideals of these singular points:  it is a smooth projective $n$-fold defined over $k$. 

It is classical that these manifolds are rational for $ n\leq 2$, and the arguments of Ueno (\cite{Ue75}) show that: 
\begin{itemize}
\item the Kodaira dimension of $X_{n,6}$ is $0$ if $n \ge 6$ and $-\infty$ if $n \le 5$,
\item
  the Kodaira dimension of $X_{n,4}$ is $0$ if $n \ge 4$ and $-\infty$ if $n \le 3$.
  \end{itemize}

   Much later Koll\'ar and Larsen (\cite{Kollar2}) showed a more general result:  if $Z$ has trivial canonical bundle and a finite group $G$ acts on $Z$, either the quotient
  $Z/G$ has Kodaira dimension 0, or it is uniruled.

 Ueno asked about separable unirationality of the manifold   $X_{3,4}$, and Oguiso asked the similar question  for $X_{n,6}$, $3 \leq n \leq 5$ .

 Interest for these open questions was revived  by Campana, who showed (\cite{Ca12}) that $X_{3,4}$ is rationally connected and asked about rationality of $X_{3,4}$; unirationality was proven 
by Catanese, Oguiso and Truong in \cite{COT13}, and later Colliot-Th\'el\`ene proved rationality in \cite{CTh13} using the conic bundle description of 
\cite{COT13}.

In the case of the Ueno manifolds, Oguiso and Truong proved 
in  \cite{OT13} that  $X_{3,6}$ is rational. 

The main result of the present paper is the following.

\begin{theorem}\label{thm1}
$X_{4,6}$ is unirational.
\end{theorem}

Our description can be useful to attack the further questions:

\begin{question}\label{quest1}
 Is  $X_{4,6}$  rational?
 
\end{question}

\begin{question}\label{quest1}
Is  $X_{5,6}$  unirational? Is  $X_{5,6}$  rational?
\end{question}

The rebirth of interest in the rationality of these manifolds stems also from complex dynamics and entropy, since these manifolds admit an action by $GL(n,\mathbb{Z})$
(and indeed by $GL(n,R_m)$, where $R_m$ is the cyclotomic ring $\mathbb{Z}[i]$, resp.  $\mathbb{Z}[\omega]$).

 In fact, $GL(n,\mathbb{Z})$ and $GL(n,R_m)$ act on the product $E(m)^n$;  and, since we divide by  a central automorphism, the action first of all descends to the quotient,
 and then it extends biregularly to 
 $X_{n,m}$ since  the resolution is just obtained by 
 blowing  up  the singular points of the quotient.

In the case of the Ueno manifolds, Oguiso and Truong proved 
in  \cite{OT13} that  $X_{3,6}$ is rational. 
They  not only proved the rationality of $X_{3,6}$, but also  showed that in this way one gets a rational variety with a primitive automorphism 
of positive entropy. Here, according to a concept introduced by De-Qi-Zhang (see \cite{Zhang}), an automorphism $f : X \rightarrow X$  is said to be birationally inprimitive
if there is a nontrivial rational fibration $\pi : X \rightarrow Y$, and a birational  automorphism $\phi$  of $Y$ such that $ \pi \circ f = \phi \circ \pi$.
De-Qi-Zhang showed that if a threefold $X$ admits a primitive birational automorphism of positive entropy, then either $X$ is a torus, or it is a $\mathbb{Q}$-Calabi-Yau manifold,
or $X$ is rationally connected.

 \begin{question}\label{quest2}
Does a similar result hold for $X_{4,6}$?
\end{question}
 \bigskip
 
 In another vein, our specific unirationality result lends itself to  more general questions. To formulate these, we need to briefly
 describe the steps of the proof, and the analogy with the case of $X_{3,4}$.

The first step of the proof  is computational, and consists in finding a minimal system of generators for the field of invariant
rational functions on $E^n_m$: here  the cases of $X_{3,4}$ and $X_{4,6}$ are treated quite similarly. 
For instance, in the case of $X_{3,6}$ one finds three generators,
hence these three elements are algebraically independent and the variety is $k$-rational.

In the case of $X_{3,4}$ we found  4 generators $t_1, t_2, u_1, u_2$ and one equation, which can be written as a diagonal quadratic form 
of the form 
$$  u_1^2 -  A(t_1, t_2)  u_2^2 - B(t_1, t_2)  = 0.$$

We thus got, birationally, a conic bundle over the projective plane, and the method of \cite{COT13}  consisted in showing that
the conic bundle has a bisection $Z$ which is rational: then the pull back of the conic bundle to $Z$
is a conic bundle with a section hence it is rational.

Colliot-Th\'el\`ene proved that the conic bundle does not have a section: in fact, if $K$ is the function field of the plane,
$A,B \in K$ and to such a diagonal conic over $K$ one associates a central algebra over $K$,
$M_{A,B} $, generated by $1,i,j, ij= - ji$ and defined by $ i^2= A, j^2 = B$.

By a general theorem the algebra is a division algebra if and only if the conic does not have any $K$-rational point
(in the contrary case $M_{A,B} \cong M(2,2,K)$). 
Moreover, two such conics are $K$-isomorphic if and only if the corresponding algebras are isomorphic
(they yield the same element of the Brauer group).

Colliot-Th\'el\`ene proved also that in this case the conic is isomorphic to one of the form 
$$  u_1^2 +  t_1  u_2^2 +  t_2  = 	1,$$
hence the function field is generated by $t_1, u_1,  u_2$ and $X_{3,4}$ is rational.

To prove here  the unirationality of $X_{4,6}$ , swe how that it is birational to a diagonal cubic surface $S$ over the function field
$K := k(t_1, t_2)$ 

$$ A(t_1, t_2)  (u_1^3-1) +  B(t_1, t_2)  (u_2^3-1)  + C(t_1, t_2) (u_3^3-1)  = 0.$$

The surface $S$ admits  therefore 27 rational points ( just let $u_j$ be a cubic root of $1$).

Then, by a theorem of B. Segre, it follows that $S$ is unirational;  we further observe here  that the degree of unirationality   is at most 6,
and we conjecture it to be at most 2.

Using other classical  results of B. Segre, Swinnerton-Dyer and Colliot-Th\'el\`ene on cubic surfaces and on diagonal cubic surfaces we show finally 
that  the surface $S$  is $K$- unirational, but it is not  $K$- rational.

We ask whether  it is possible, like it was done for the conic bundle case, to change the cubic surface birationally and 
obtain equations which imply the rationality of  $X_{4,6}$.

Observe that the coefficients $A(t_1, t_2), B(t_1, t_2)  ,C(t_1, t_2)  $  correspond to a very special system of plane cubics,
yielding the Del Pezzo surface of degree 2 which is the double cover of 
$\mathbb{P}^2$  branched over a complete quadrilateral.
 
 In particular, in  the course of our proof, we show that our variety $X_{4,6}$ is birational to a hypersurface $X$ of bidegree $(3,3)$
 inside $\PP^2 \times \PP^3$. 
 
 Viewing $X$ as a cubic surface $S$  over the function field $K:= k(\PP^2) = k(t_1, t_2)$ it follows, by  the cited theorem of Beniamino Segre,
 that $X$ is $K$-unirational if the surface $S$ admits a $K$-rational point, and that there are many such surfaces
 which are $K$-rational (note that then $X$ is  $k$-rational).
 
 \begin{question}\label{quest3}
Let $X$ be a very general hypersurface of bidegree $(3,3)$
 inside $\PP^2 \times \PP^3$. 
 
Is $X$ unirational? Is  $X$  rational?
\end{question}

\section{Proof of Theorem (\ref{thm1})}

Here $n = 4$ and we set $Z : = Z_4 (6)$ and $X : = X_{4,6}$, $g := g_6$. 

We write  $Z : = Z_4 (6) = C_1 \times C_2   \times C_3 \times C_4$, and let 
 $g : = g_6$ be the diagonal action $ g(1)  \times g(2)    \times g(3) \times g(4) $, where
the curves $(C_i, g(i)$ ($i =1,2,3, 4$) are birationally equivalent to $(E'_{\eta}, g_6)$.

 Hence we view  $C_i$ as birational to the singular curve $C_i^0$ in the affine space ${\AAA}^2 = {\rm Spec}\, k[X_i, Y_i]$, and $g(i)$
  as the automorphism of $C_i^0$  defined by
$$C_i^0 : =\{ (X_i, Y_i) | Y_i^6 = X_i^2(X_i -1)\} \,\, ,\,\, g_i^*Y_i = -\omega Y_i\,\, ,\,\, g_i^*X_i = X_i\,\, .$$

The affine coordinate ring $k[C_i^0]$ of $C_i^0$ is 
$$k[C_i^0] = k[X_i, Y_i]/(Y_i^6 -X_i^2(X_i -1))\,\, .$$ 
We set  $x_i : = X_i\, {\rm mod}\, (Y_i^6 -X_i^2(X_i -1)), y_i : = Y_i\, {\rm mod}\, (Y_i^6 -X_i^2(X_i -1))$. Then $y_i^j$ ($0 \le j \le 5$) form a free $k[x_i]$-basis of $k[C_i^0]$ over a polynomial ring $k[x_i]$ and therefore $x_i^my_i^j$ ($0 \le j \le 5$, $0 \le m$) form a free $k$-basis of $k[C_i^0]$. 

Then $(Z, g)$ is birationally equivalent to the affine fourfold 
$$V  := C_1^0 \times C_2^0 \times \times C_3^0 \times C_4^0$$
with automorphism $g = (g(1), g(2), g(3), g(4))$, and with affine coordinate ring
$$k[V] = k[C_i^0] \otimes k[C_2^0] \otimes k[C_3^0] \otimes k[C_4^0]\,\, .$$
The subring $k[x_1, x_2, x_3, x_4]$ of $k[V]$ is a polynomial ring with four free variables $x_1$, $x_2$, $x_3$, $x_4$ and $k[V]$ is a free $k[x_1, x_2, x_3, x_4]$-module with free basis  
$$y_1^{m_1}y_2^{m_2}y_3^{m_3}y_4^{m_4}\,\, {\rm ,\,\, where\,\,}\,\, 0 \le m_i \le 5\,\, .$$
 
The rational function field $k(Z)$ of $Z$ is
$$k(Z) = k(V) = k(x_1, x_2, x_3, x_4, y_1, y_2, y_3, y_4)\,\, .$$
In both $k[V]$ and $k(Z) =k(V)$, we have 
\begin{equation}\label{I} \,\, y_i^6 = x_i^2(x_i -1)\,\, ,\end{equation} 
\begin{equation}\label{II} \,\, g^*y_i = -\omega y_i\,\, ,\,\, g^*x_i = x_i\,\, .\end{equation} 
Here and hereafter each equation shall be viewed as an  equation in $k(V)$. 

The affine coordinate ring $V/\langle g \rangle$ is $k[V]^{g^*}$, the invariant subring of $k[V]$. Thus by (\ref{II}),  
\begin{equation}\label{III}\,\, k[V]^{g^*} = k[x_1, x_2, x_3, x_4][y_1^{m_1}y_2^{m_2}y_3^{m_3}y_4^{m_4} | \sum_{i=1}^{4} m_i \equiv 0 \, (\rm{mod}\, 6)]\,\, .
\end{equation} 
If $\sum_{i=1}^{4} m_i = 6k$, then
$$y_1^{m_1}y_2^{m_2}y_3^{m_3}y_4^{m_4} = (y_1^6)^k (\frac{y_2}{y_1})^{m_2}(\frac{y_3}{y_1})^{m_3}(\frac{y_4}{y_1})^{m_4} = (x_1^2(x_1-1))^k(\frac{y_2}{y_1})^{m_2}(\frac{y_3}{y_1})^{m_3}(\frac{y_4}{y_1})^{m_4}\,\, .$$
Note that $k(X) = Q(k[V]^{g^*})$, the  field of fractions of $k[V]^{g^*}$. 

Hence  
\begin{lemma}\label{lem1}
$$k(X) = k(x_1, x_2, x_3, x_4, t_2 := \frac{y_2}{y_1}, t_3 := \frac{y_3}{y_1}, 
t_4 := \frac{y_4}{y_1})$$
with relations precisely 
$$t_i^6 = \frac{x_i^2(x_i-1)}{x_1^2(x_1 - 1)}\,\, {\rm ,\,\, where\,\,}\,\, i = 2, 3, 4\,\, .$$
\end{lemma} 
\begin{proof} We only need to observe that the three relations above are all the relations. Since $[k(V) : k(x_1, x_2, x_3, x_4)] = 6^4$ and $[k(V) : k(X)] = 6$, it follows that 
$$[k(X): k(x_1, x_2, x_3, x_4)] = 6^3\,\, .$$ Thus, the equation above for $i=2$ is the minimal equation of $t_2$ over $k(x_1, x_2, x_3, x_4)$, the equation above for $i=3$ is the minimal equation of $t_3$ over $k(x_1, x_2, x_3, x_4)(t_2)$ and the equation above for $i=4$ is the minimal equation of $t_4$ over $k(x_1, x_2, x_3, x_4)(t_2, t_3)$, as desired.
 
\end{proof} 

Define:
\begin{equation}\label{IV} \,\, u_2 := \frac{x_1}{x_2}t_2^3\,\, ,\,\, u_3 := \frac{x_1}{x_3}t_3^3\,\, ,\,\, u_4 := \,\, \frac{x_1}{x_4}t_4^3\,\, .
\end{equation}Then 
\begin{lemma}\label{lem2}
$$k(X) = k(x_1, u_2, u_3, u_4, t_2, t_3, 
t_4)$$
with relations precisely 
$$\frac{u_i}{t_i^3} = \frac{x_1}{u_i^2(x_1-1)+1}\,\,{\rm ,\,\, where}\,\, i = 2,3, 4\,\, .$$
\end{lemma}
\begin{proof}
By the definition 
$$u_i := \frac{x_1}{x_i}t_i^3\,\, ,$$
the equation in Lemma (\ref{lem1}) is 
$$u_i^2 = \frac{x_i -1}{x_1 -1}\,\, .$$
Thus 
$$x_i = u_i^2(x_1 -1) +1\,\, .$$
Therefore
$$k(X) = k(x_1, u_2, u_3, u_4, t_2, t_3, t_4)$$
with relations precisely
$$u_i (= \frac{x_1}{x_i}t_i^3) = \frac{x_1}{u_i^2(x_1 -1) +1}t_i^3\,\, .$$
Dividing both sides by $t_i^3 \not= 0$ in $k(V)$, we complete the proof.

\end{proof}

\begin{lemma}\label{lem3}
$$k(X) = k(u_2, u_3, u_4, t_2, t_3, 
t_4)$$
with relations precisely 
$$\frac{u_2 -t_2^3}{u_2^3 -u_2} = \frac{u_3 -t_3^3}{u_3^3 -u_3} = \frac{u_4 -t_4^3}{u_4^3 -u_4}\,\, .$$
\end{lemma}
\begin{proof} The equations in Lemma (\ref{lem2}) are linear with respect to $x_1$. 

In fact, by clearing the demoninator
$$u_i^3(x_1-1) +u_i = t_i^3x_1\,\, ,$$
that is (by adding $-t_i^3x_1 +u_i^3 -u_i$ to  both sides), 
$$(u_i^3 -t_i^3)x_1 = u_i^3 - u_i\,\, .$$
Observe that $u_i^3 - u_i \not= 0$ in $k(V)$. Thus, this is equivalent to
$$\frac{1}{x_1} = \frac{u_i^3 -t_i^3}{u_i^3 - u_i}\,\, ,$$
that is, (by adding $-1$ to  both sides), 
$$\frac{1}{x_1} -1 = \frac{u_i -t_i^3}{u_i^3 - u_i}\,\, .$$
This implies the result.

\end{proof}

Observe that $u_i \not= 0$ in $k(V)$ and define:
\begin{equation}\label{V} \,\, v_2 := \frac{1}{u_2}\,\, ,\,\, w_2 := \frac{t_2}{u_2}\,\, ,\,\, v_3 := \frac{1}{u_3}\,\, ,\,\, w_3 := \frac{t_3}{u_3}\,\, ,\,\, 
v_4 := \frac{1}{u_4}\,\, ,\,\, w_4 := \frac{t_4}{u_4}\,\, .\end{equation}
\begin{lemma}\label{lem4}
$$k(X) = k(v_2, v_3, v_4, w_2, w_3, 
w_4)$$
with relations precisely 
$$(v_3^2 -1)(w_2^3-1) = (v_2^2 -1)(w_3^3 -1)\,\, , \,\, (v_4^2 -1)(w_2^3-1) = (v_2^2 -1)(w_4^3 -1)\,\, .$$
\end{lemma}
\begin{proof}
By the definition, $v_i = 1/u_i$, $w_i = t_i/u_i$ and $u_i = 1/v_i$, $t_i = w_iu_i = w_i/v_i$, it follows that
$$ k(u_2, u_3, u_4, t_2, t_3, 
t_4) = k(v_2, v_3, v_4, w_2, w_3, 
w_4)\,\, .$$
Observe that 
$$\frac{u_i -t_i^3}{u_i^3 -u_i} = \frac{(1/u_i)^2 - (t_i/u_i)^3}{1 - (1/u_i)^2} = \frac{v_i^2 -w_i^3}{1-v_i^2} = -1 + \frac{1-w_i^3}{1-v_i^2}\,\, .$$
Hence the precise relations in Lemma (\ref{lem3}) are rewritten in terms of the new variables as
$$\frac{1-w_2^3}{1-v_2^2} = \frac{1-w_3^3}{1-v_3^2} = \frac{1-w_4^3}{1-v_4^2}\,\, .$$
By clearing the demoninators, we obtain the precise relations that we  claimed.

\end{proof}

Observe that $v_i -1 \not= 0$ in $k(V)$ and define:
\begin{equation}\label{VI} \,\, s_3 := \frac{v_3-1}{v_2-1}\,\, ,\,\, s_4 := \frac{v_4 -1}{v_2 -1}\,\, .\end{equation}
\begin{lemma}\label{lem5}
$$k(X) = k(s_3, s_4, w_2, w_3, 
w_4)$$
with relations precisely 
$$(s_3-s_4)s_3s_4(w_2^3-1) -(s_3-1)s_3(w_4^3-1) +(s_4-1)s_4(w_3^3-1) = 0\,\, .$$
\end{lemma}
\begin{proof} The defining equation of $s_i$ is linear in both $s_i$ and $v_i$, hence  it follows that
$$k(X) = k(v_2, v_3, v_4, w_2, w_3, 
w_4) = k(v_2, s_3, s_4, w_2, w_3, w_4)\,\, .$$
Since 
$$ v_i + 1 = (v_i -1) + 2\,\, ,\,\, v_i^2 -1 = (v_i -1)(v_i+1) = s_i(v_2-1)(s_i(v_2-1) +2)\,\, $$
and $v_2 -1 \not= 0$ in $k(V)$, 
the relations in Lemma (\ref{lem4}) are precisely
$$s_3(s_3(v_2-1)+2)(w_2^3-1) = (v_2+1)(w_3^3-1)\,\, ,$$
$$s_4(s_4(v_2-1)+2)(w_2^3-1) = (v_2+1)(w_4^3-1)\,\, .$$
Both are linaer in terms of $v_2$, more explicitly, these two equations are equivalent to
$$(s_3^2(w_2^3 -1) -(w_3^3-1))(v_2-1) = -2s_3(w_2^3-1) + 2(w_3^3 -1)\,\, .$$
$$(s_4^2(w_2^3 -1) -(w_4^3-1))(v_2-1) = -2s_4(w_2^3-1) + 2(w_4^3 -1)\,\, .$$
Observe that 
$$-s_i^2(w_2^3-1) + (w_i^3 -1) \not= 0$$ in $k(V)$ and recall that $k$ is not of characteristic $2$. Then 
these two equations are equivalent to
$$\frac{v_2 -1}{2} = \frac{-s_3(w_2^3-1) + (w_3^3 -1)}{s_3^2(w_2^3 -1) -(w_3^3-1)} = \frac{-s_4(w_2^3-1) + (w_4^3 -1)}{s_4^2(w_2^3 -1) -(w_4^3-1)}\,\, .$$
Thus 
$$k(X) = k(s_3,s_4, w_2, w_3, w_4)\,\, ,$$
with the above precise relation, that is (by taking the inverse and multiply by $-1$)
$$\frac{s_3^2(w_2^3 -1) - (w_3^3-1)}{s_3(w_2^3-1) - (w_3^3 -1)} 
= \frac{s_4^2(w_2^3 -1) - (w_4^3-1)}{s_4(w_2^3-1) - (w_4^3 -1)}\,\, .$$
Observe that 
$$s_i^2(w_2^3 -1) - (w_i^3-1) = (s_i^2 -s_i)(w_2^3 -1) + (s_i(w_2^3 - 1) - (w_i^3 -1))\,\, .$$ 
Thus the equation above is equivalent to
$$\frac{(s_3^2-s_3)(w_2^3 -1)}{s_3(w_2^3-1) - (w_3^3 -1)} +1 = \frac{(s_4^2-s_4)(w_2^3 -1)}{s_4(w_2^3-1) - (w_4^3 -1)} +1\,\, ,$$
whence, equivalent to
$$\frac{(s_3^2-s_3)}{s_3(w_2^3-1) - (w_3^3 -1)} = \frac{(s_4^2-s_4)}{s_4(w_2^3-1) - (w_4^3 -1)}\,\, ,$$
by $w_2^3 -1 \not= 0$ in $k(V)$. 
By clearing demoninators, the last equation is equivalent to
$$(s_3^2-s_3)(s_4(w_2^3-1) - (w_4^3 -1)) - (s_4^2-s_4)(s_3(w_2^3-1) - (w_3^3 -1)) = 0\,\, ,$$
which is nothing but the equation claimed (just make the equation as an equation with respect to $w_i^3 -1$). 

\end{proof}

To proceed with  the proof of Theorem (\ref{thm1}),
consider the affine hypersurface $H$ defined by
$$(s_3-s_4)s_3s_4(w_2^3-1) -(s_3-1)s_3(w_4^3-1) + 
(s_4-1)s_4(w_3^3-1) = 0\,\, $$
in the affine space ${\AAA}^5$ with affine coordinates $(s_3, s_4, w_2, w_3, w_4)$. 

Since $X$ is of dimension $4$, Lemma (\ref{lem5}) means that $X$ is birational to $H$.

\begin{rem}

 The projection $\pi : H \to {\AAA}^2$ defined by $$(s_3, s_4, w_2, w_3, w_4) \mapsto (s_3, s_4)$$ makes $H$ a  fibration  of cubic surfaces
 over the affine space ${\AAA}^2$ with affine coordinates $(s_3, s_4)$. 
 
 Let $\eta$ be the 
generic point of 
${\AAA}^2$ and let $H_{\eta}$ be  the generic fiber. Then the projective completion $\overline{H}_{\eta}$ of $H_{\eta}$ in ${\mathbb P}_{\eta}^3$ is a smooth cubic surface $S$  over the field $k(s_3, s_4)$ (by the Jacobian criterion and the fact that $k$ is not of characteristic $3$) with a rational point $(1,1,1) \in \overline{H}_{\eta}(k(s_3, s_4))$. 

\end{rem}

The above remark shows  that it is sufficient to show that $\overline{H}_{\eta}$ is unirational over $K: = k(s_3, s_4)$ (i.e., there is a dominant rational map ${\mathbb P}_{\eta}^2 --\to \overline{H}_{\eta}$ over $k(s_3, s_4)$): because then, $k(s_3, s_4)$ being  purely transcendental, this means that there is a dominant rational map ${\mathbb P}_{k}^{4} --\to X$ over $k$, that is, $X$ is unirational over $k$. 

The $K$-unirationality of $S= \overline{H}_{\eta}$ follows 
  by a theorem of Segre (see \cite{Seg43}, see also the extension to higher dimensional cubic hypersurfaces done by Koll\'ar in \cite[Theorem 1]{Ko02}),
  asserting that a smooth cubic surface with a $K$-rational point is  $K$-unirational.
  
  In the next section we shall give an upper bound of the $K$-unirationality degree, and show that $S$ is not $K$-rational.
  
\section{Diagonal cubic surfaces and their rationality}

We begin this  section  recalling several known results, which immediately imply the claimed assertion  that our cubic surface $S$ is not $K$-rational.

Let $S$ be a nonsingular cubic surface defined over a field $K$: then Swinnerton-Dyer completed earlier results by B.Segre in 
\cite{Seg51} showing 

\begin{thm}{\bf (Swinnerton-Dyer \cite{SD70})}
A smooth cubic surface $S$ defined over a field $K$ is birational to $\PP^2_K$ if and only if 

1) it has a $K$-rational point

and

2) $S$  contains a set $\Sigma_n$ of pairwise disjoint lines which is defined over $K$ (i.e., invariant under the Galois group  $Gal(\bar{K},K)$)
and has cardinality $n \in \{2,3,6\}$.

\end{thm}

If one drops the second condition, then one has 

\begin{thm}{\bf (Segre \cite{Seg43})}
A smooth cubic surface $S$ defined over a field $K$ is $K$-unirational if and only 
 it has a $K$-rational point.
\end{thm}

Recall now that a {\bf diagonal cubic surface $S$} is a cubic surface 
$$  S \subset \PP^3_K , \ S = \{ (x_1,x_2,x_3, x_4) | \sum_{i=1}^4  a_i x_i^3 = 0 \},$$
and one says that $S$ is defined over $K$ if $a_i \in K, \forall i=1, \dots, 4$.

Observe that, in the case where the field $K$ is moreover algebraically closed, then such a diagonal cubic surface $S$ is projectively
equivalent to the Fermat cubic surface. For diagonal cubic surfaces it is easy to find the lines lying on it,
moreover they have a special geometry; recall for this the definition of {\bf Eckardt points}: these are the  points $P$ where the tangent plane intersects
the surface $S$ in a set of three lines passing through $P$.

\begin{proposition}
Let $K$ be a field of characteristic $\neq 2,3$ and containing a primitive third root of unity $\omega$.
Let $S$ be, as above, a diagonal cubic surface defined over $K$ (i.e., $a_i \in K$). 
$$  S \subset \PP^3_K , \ S = \{ (x_1,x_2,x_3, x_4) | \sum_{i=1}^4  a_i x_i^3 = 0 \}.$$

Then the 27 lines of $S$, defined over a field extension of $K$, are the three sets of 9  lines
obtained, for each partition $\{1,2,3,4\} = \{i,j\} \cup \{h,k\}$,
by the equations
$$ a_i x_i^3 +  a_j x_j^3 =  a_h x_h^3 +  a_k x_k^3 = 0,$$
i.e., 
$$ x_i = \la_{ij} x_j , \  x_h = \la_{hk} x_k , a_i \la_{ij}^3  + a_i=0 , \ a_h  \la_{hk}^3 +  a_k = 0.$$

Moreover, the surface possesses exactly 18 Eckardt points, defined on the algebraic closure of $K$
by the equations $x_i=x_j=0,1 \leq  i,j \leq 4$.
\end{proposition}
\Proof
The first assertion being clear, we prove the second assertion.
Since the coefficients $a_i \neq 0$, due to the smoothness of $S$, the complete intersection $\Ga$ of $S$ with the Hessian surface $H_S = \{x_1 x_2 x_3 x_4 =0 \}$
of $S$ consists of 4 smooth plane cubics, which intersect in the nine points $x_i = x_j = 0, a_h x_h^3 + a_k x_k^3=0$,
which are therefore the singular points of $\Ga$. Notice that an Eckardt point must be a singular point of $\Ga$; conversely, 
at one such point the tangent plane is the plane
$$  \{ y |  y_h a_h x_h^2 + y_k a_k x_k^2=0\}. $$
Since $a_h x_h^3 + a_k x_k^3=0$, the intersection of $S$ with the tangent plane is the union of the three lines
$$  \{ y | y_k x_h = y_h x_k, \   a_i y_i^3 +  a_k y_k^3=0\}. $$
\qed

\begin{rem}
The maximum number of Eckardt points that a smooth cubic surface can have is exactly 18. 

This can be shown using the  model of the cubic surface as the blow-up of the plane in six points not lying on a conic.
The analysis using 
the parabolic curve $\Ga$, which  is a complete intersection of type $(3,3)$, hence has  arithmetic genus equal to 19,
and  degree equal to 12, seems more complicated. 

\end{rem}

\begin{thm}
Let $K$ be a field of characteristic $\neq 2,3$ and containing a primitive third root of unity $\omega$.
And let $S$ be a smooth diagonal cubic surface defined over $K$ and with equation
$$    S = \{ (x_1,x_2,x_3, x_4) | a_1 (x_1^3-x_4^3) +  a_2 (x_2^3-x_4^3) +  a_3 (x_3^3-x_4^3) = 0 \}.$$
Then there is a dominant rational map $\PP^2_K \ra S$ of degree at most 6.
\end{thm}

\Proof
Observe first of al that $S$ is smooth if and only if $a_1 a_2 a_3 (a_1 + a_2 +a_3) \neq 0$.

Second, $S$ contains the 27 points whose coordinates are cubic roots of $1$. As we have shown, these are not Eckardt points.

Assume now that $P \in S$ is a $K$-rational point, so that the curve
$$C_P = S \cap (T_PS)  $$
intersection of the surface with the tangent plane in $P$ is defined over $K$, and has a singular point at $P$.

Case 1): $C_P$ is irreducible.

Then $C_P$ is birational to $\PP^1_K$, and we have a birational parametrization $\la \in K \mapsto P_{\la} \in C_P$.
Observe that the plane $(T_PS) $ meets the lines contained in $S$ in a finite set, hence through the general point
$P_{\la} $ passes no line contained in $S$.

Then we let $D_{\la} = S \cap (T_{ P_{\la}}S)  $,  which is in general an irreducible cubic, singular in $ P_{\la}$:
hence there is a dominant rational map $\psi : \PP^1_K \times \PP^1_K \ra S$.

In this case the degree of $\psi$ is equal to 6.

In fact,  $C_P$  maps to a curve $C'$  of degree 6 on the dual surface  $S^{\vee}$ of $S$.

And the condition for a point $z \in S$ that $z \in T_yS$  means that the dual plane $z^*$ contains the point $(T_yS)^* $.

So, if we intersect the curve $C'$ with the plane $z^*$ we obtain 6 points:  this show that the degree of the map $\psi$  is 6.

Case 2):  $C_P$ has two components, a line $L$ and a conic $Q$.

In this case $L$ is $K$-rational.  In general, if a cubic $S$ contains such a $K$-rational line, then the pencil of planes $\pi_t \supset L$
yields a  dominant rational map $\phi : \PP^1_K \times \PP^1_K \ra S$ of degree 2. Because for each point $P' \in L$ the tangent plane
in $P'$ intersects $S$ in $ L \cup Q_{P'}$, and for general $P' \in L$, $K$-rational,  the conic $Q_{P'}$ is irreducible and contains $P'$, hence
we obtain a dominant rational map $\phi : \PP^1_K \times \PP^1_K \ra S$. For a general $ x \in S$, the plane spanned by $x$ and $L$
intersects $S$ as $ L \cup Q$, and $Q$ intersects $L$ in two points: hence the degree of $\phi$ equals 2.

We shall however see in theorem \ref{diagonal} that for a diagonal cubic surface with a $K$-rational point the existence of a $K$- rational line implies the 
birationality of $S$ with $\PP^2_K$.

Case 3): $C_P$ consists of three lines.

Then, since $P$ is singular for $C_P$, at least two of these lines contain $P$. If there is one of the three lines which does not contain $P$,
then this line is $K$-rational, and we are done as in case 2). There remains only

Case 4):  $C_P$ consists of three lines passing through $P$, and none of them is $K$-rational. 

Then $P$ is classically called  an {\bf Eckardt point} and since $P$ is $K$-rational  the three  lines passing through $P$
form a Galois orbit.

We conclude that the number of such points is at most $9$, hence we have shown that 
 there is a dominant rational map $\psi: \PP^2_K \ra S$ of degree at most 6.

\qed

The following theorem was first stated by  Segre in \cite{Seg43} and \cite{Seg51}, and was then also proven by  
Colliot-Th\'elene, Kanevsky, Sansuc in \cite{CTKS}.

\begin{thm}{\bf (Segre, Colliot-Th\'elene, Kanevsky and Sansuc)}\label{diagonal}
Let $K$ be a field of characteristic $\neq 2,3$ and containing a primitive third root of unity $\omega$.
And let $S$ be a smooth diagonal cubic surface defined over $K$ and with equation
$$    S = \{ (x_1,x_2,x_3, x_4) | a_1 x_1^3+  a_2 x_2^3 +  a_3 x_3^3 + a_4 x_4^3= 0 \}.$$

Then $S$ is $K$-rational if and only if it has a $K$-rational point and 
there is a permutation $a,b,c,d$ of the four coefficients $a_1, \dots, a_4$ such that $ac/bd$ has a cubic root in $K$.
\end{thm}

\Proof

Without loss of generality we may assume that $a_1=1$.

Recall that for a diagonal cubic surface $S$ the 27 lines are grouped into 3 subsets $\sR_i$ of 9 elements,
corresponding to indices $i=2,3,4$.

In fact, write $\{1,2,3,4\} = \{1,i\} \cup \{j,k\}$: then any line has the form
$$ x_1 + \be_i x_i = 0, x_j + \be_{j,k} x_k = 0, $$ 
where $\be_i ^3 = a_i$, $ \be_{j,k} ^3 = a_k / a_j.$

This shows that $S$  has a $K$-rational line  if and only if 
there is a permutation $a,b,c,d$ of the four coefficients $a_1, \dots, a_4$ such that $a/b$ and $c/d$ have a cubic root in $K$.
 
Consider now the field extension $K'$ generated by the cubic roots of the coefficients $a_2, a_3, a_4$
(recall that we are assuming   that $a_1=1$).

Then  the Galois group $G: = Gal (K', K)$ is $(\ZZ/3)^m, m=1,2,3$. Each set $\sR_i$ is a union of Galois orbits,
and it is a single Galois orbit if the field extension $K_i$ generated by $ a_i, a_k / a_j$ has degree 9.

We already observed that $K_i = K$ for some $i$ if and only if there  is a $K$-rational line.

In this case two lines
$$ x_1 + \be_i x_i = 0, x_j + \be_{j,k} x_k = 0, $$ 
$$ x_1 + \be'_i x_i = 0, x_j + \be'_{j,k} x_k = 0, $$
are $K$-rational and skew if $  \be_i \neq  \be_i', \be_{j,k} \neq \be_{j,k} '$.

In this way we have found three pairwise disjoint and $K$-rational lines, and $S$ is rational
by the criterion of Segre and
 Swinnerton-Dyer.

Now, if there is no $K$-rational line, then all the Galois orbits inside $\sR_i$ have either cardinality 9, or
are three orbits of cardinality 3. But in this case the field extension $K_i$ generated by $ a_i, a_k / a_j$ has degree 3.
If $a_i \in K$ or $a_k / a_j \in K$ then the orbits consist of incident lines, and one cannot apply the criterion
 of Segre and Swinnerton-Dyer. The only possibility is that both $a_i \notin K, a_k / a_j \notin K$
 but $a_k/a_j a_i \in K$.
 
 Then we get a Galois orbit of 3 pairwise disjoint lines, and we need to have a $K$-rational point in order to apply
 the criterion of Segre and
 Swinnerton-Dyer.

\qed

We can now show that our surface $S$ is not $K$-rational.

\begin{thm}
Let $K := k (s_3, s_4)$ and let $S$ be the diagonal cubic surface of equation
$$(s_3-s_4)s_3s_4(x_1^3-x_4^3) -(s_3-1)s_3(x_2^3-x_4^3) + 
(s_4-1)s_4(x_3^3-x_4^3) = 0. $$
Then $S$ is not $K$-rational.
\end{thm}

\Proof
The four coefficients are $(s_3-s_4)s_3s_4, -(s_3-1)s_3, (s_4-1)s_4$ and their sum $f= (s_3-s_4)s_3s_4 -(s_3-1)s_3 + (s_4-1)s_4$.

Take  a permutation $a,b,c$ of the first three coefficients and consider $ f a/ bc$.
This  is a fraction with relatively prime numerator and denominator, and such that the denominator is not a cube;
hence  $ f a/ bc$ is not a cube in $K$, and by the previous theorem \ref{diagonal} $S$ is not $K$-rational.  

\qed

\section{Special geometry  of $X_{4,6}$}

Let us analyse the equation of the hypersurface $H$.

It is a diagonal cubic surface in the variables $w_i$, with equation 
$$a (s_3,s_4) (w_2^3-1) + b (s_3,s_4) (w_4^3-1) + 
c (s_3,s_4) (w_3^3-1) = 0\,\, $$
where the coefficients $a,b,c$ are cubic polynomials
$$a:= (s_3-s_4)s_3s_4, \  b : =  - (s_3-1)s_3 , \ c : =   (s_4-1)s_4\,\, $$

The coefficients $a,b,c$ yield a rational map $\phi$ of the plane $\PP^2$  into $\PP^2$, given by the system of cubics which,
in homogeneous coordinates $(u_2, u_3, u_4)$ (here $s_3 =   u_3/u_2, s_4 =   u_4/u_2$)
reads out as
$$ \phi (u_2, u_3, u_4) = [(u_3-u_4)u_3u_4 :   (u_2-u_3)u_2 u_3 :   (u_4-u_2)u_4 u_2].$$
 
The cubic polynomials $a,b,c$ are products of linear forms and the indeterminacy locus of $\phi$ is the
set of $7$ points 
$$\{(1,1,1), (0,1,1), (1,0,1), (1,1,0), (1,0,0), (0,1,0), (0,0,1).\}$$
Hence $\phi$ is a double covering, and induces a birational involution on the source $\PP^2$,
which is classically called the {\bf Geiser involution}.

These seven base points are simple base points and blowing up the plane $\PP^2$ at them one obtains a Del Pezzo
surface of degree $2$, $F$, which is a double cover of the plane branched on a curve $\sB$ of degree $4$.

In order to compute effectively $\sB$, we calculate the ramification divisor of $\phi$, which is given by the
determinant of the Jacobian matrix of $\phi$.

Thus the equation $R(u_2, u_3, u_4)$ of the ramification divisor is the determinant of the following matrix
$$
\begin{pmatrix}
0 & 2 s_3 s_4 - s_4^2& s_3^2- 2 s_3 s_4\\
s_4^2 -  2 s_2 s_4&    0 & 2 s_2 s_4 - s_2^2 \\
2 s_2 s_3 - s_3^2 & s_2^2 - 2 s_2 s_3 & 0\\
\end{pmatrix}
$$

An elementary calculation yields
$$R(u_2, u_3, u_4) = 6  u_2 u_3 u_4 (u_2 - u_3)  (u_3 - u_4)  (u_4 - u_2) . $$

Thus the ramification divisor consists of six lines, which are contracted to the $6$ points 
$$\{ (0,1,-1), (1,0,1), (1,0,-1), (1,0,0), (0,1,0), (0,0,1).\}$$

The six lines of the ramification divisor intersect in the four points 
 $$ \{e, e_1, e_2, e_3 \} = \{ (1,1,1), (1,0,0), (0,1,0), (0,0,1)\},$$
 through each of which three of the lines meet, and in the further three points
 $$ \{ P_1, P_2, P_3\} = \{ (0,1,1), (1,0,1), (1,1,0)\}, $$
 where only two lines of the sixtuple pass.

In the Del Pezzo surface $F$ the strict transforms of the six lines are $(-2)$ (smooth) rational curves,
the points  $e, e_1, e_2, e_3$ yield $(-1)$ (smooth)  rational curves, which are seen to be
part of the ramification divisor of the degree $2$ morphism
$$ \varphi : F \ra \PP^2$$
since the strict transform of $R$ contains these curves with multiplicity $3$.
Since the strict transform of $R$ contains the $(-1)$ (smooth)  rational curves, blow up 
of the points $P_i$,  with multiplicity $2$, these curves are not in the ramification divisor
of $\varphi$.

The conclusion is that  the branch locus is the image of the $4$ curves blow up of  $e, e_1, e_2, e_3$,
i.e.,
$$ \sB = \{ (a,b,c) | a b c (a+b+c) = 0 \}. $$ 

Hence the Del Pezzo surface $F$,  the double cover of $\PP^2$ branched on $\sB$,
is contained in the line bundle $\LL$ over $\PP^2$ whose sheaf of sections is $\hol_{\PP^2}(2),$
and is defined there by the equation
$$ s^2 =  a b c (a+b+c). $$

\section{Some  remarks in the case $n=5$ }

Our computation for $n=4$ is quite symmetric with respect to the variables with  index $\ge 3$. 
So  exactly the same computation yields:

\begin{lemma}\label{lem1}
$$k(X_5) = k(s_3, s_4, s_5, w_2, w_3, 
w_4, w_5)$$
with relations precisely 
$$(s_3-s_4)s_3s_4(w_2^3-1) = (s_3-1)s_3(w_4^3-1) - (s_4-1)s_4(w_3^3-1)\,\, ,$$
$$(s_3-s_5)s_3s_5(w_2^3-1) = (s_3-1)s_3(w_5^3-1) - (s_5-1)s_5(w_3^3-1)\,\, .$$
In other words, $X_5$ is birational to the above complete intersection  in ${\AAA}^7$ with affine coordinate 
$$(s_3, s_4, s_5, w_2, w_3, w_4, w_5)\,\, $$
defined by the above equations of bidegree $(3,3)$.
\end{lemma}

{\bf Aknowledgement.} The first two authors  would like to  thank  Fabio Perroni for his invitation to the conference SYMKADYN held at the SISSA Trieste in September 2014: there the initial idea of this note grew up; and also KIAS Seoul, where most part of the paper was written while the first author
was visiting as KIAS scholar.


\begin{thebibliography}{COT13}

\bibitem[Ca12]{Ca12}{F. Campana,} \textit{Remarks on an example of K. Ueno}, Series of congress reports, `Classification of algebraic varieties,' Editors: C. Faber, G. van der Geer, and E. Looijenga, European Mathematical Society, 
2011, 115--121.  


\bibitem[COT13]{COT13}  Catanese, Fabrizio; Oguiso, Keiji; Truong, Tuyen Trung, \textit{Unirationality of Ueno-Campana's threefold},    Manuscripta Math. 145 (2014), no. 3-4, 399--406, arXiv:1310.3569. 

\bibitem[CTh13]{CTh13}{J.-L. Colliot-Th\'el\`ene,} \textit{Rationalit\'e d'un 
fibr\'e en coniques}, arXiv:1310.5402.

\bibitem[CT-K-S87]{CTKS}
Colliot-Th\'el\`ene, Jean-Louis; Kanevsky, Dimitri; Sansuc, Jean-Jacques
 \textit{Arithm\'etique des surfaces cubiques diagonales.} Diophantine approximation and transcendence theory (Bonn, 1985), 1--108, 
Lecture Notes in Math., 1290, Springer, Berlin, 1987. 

\bibitem[OT13]{OT13}{K. Oguiso and T. T. Truong,} \textit{Explicit Examples of rational and Calabi-Yau threefolds with primitive automorphisms 
of positive entropy}, J. Math. Sci. Univ. Tokyo 22 (2015), 361--385, ArXiv:1306.1590. 

\bibitem[Ko02]{Ko02} { Koll\'ar, J\'anos} \textit{Unirationality of cubic hypersurfaces}, J. Inst. Math. Jussieu {\bf 1}  (2002) 467--476.

\bibitem[Ko-La09]{Kollar2}
Koll\'ar, J\'anos,
Larsen, Michael
\textit{Quotients of Calabi-Yau varieties.}
in :`Algebra, arithmetic, and geometry. `In honor of Y. I. Manin on the occasion of his 70th birthday.' Vol. II. Boston, MA: Birkh\"auser Progress in Mathematics 270, 179--211 (2009).

\bibitem[Seg43]{Seg43}
Segre, Beniamino,
 \textit{ A note on arithmetical properties of cubic surfaces. }
J. London Math. Soc 18, (1943). 24--31. 

 \bibitem[Seg43]{Seg43}
Segre, Beniamino,
 \textit{Arithmetic upon an algebraic surface.}Bull. Amer. Math. Soc. 51, (1945). 152--161. 

\bibitem[Seg51]{Seg51}
Segre, Beniamino,
 \textit{ On the rational solutions of homogeneous cubic equations in four variables.} Math. Notae 11, (1951). 1--68.
 
\bibitem[SD70]{SD70}
 Swinnerton-Dyer, H. P. F.
 \textit{ The birationality of cubic surfaces over a given field. }
Michigan Math. J. 17 1970 289--295. 

\bibitem[Ue75]{Ue75}{K. Ueno,} \textit{Classification theory of algebraic varieties and compact complex spaces. 
Notes written in collaboration with P. Cherenack,} Lecture Notes in Mathematics, {\bf 439} Springer-Verlag, Berlin-New York, 1975. 


\bibitem[Zhang09]{Zhang} Zhang, De-Qi, \textit{Dynamics of automorphisms on projective complex manifolds}, J.Diff. Geometry {\bf 82}, (2009), 691-722.

\end{thebibliography}
\end{document}